\documentclass[11pt, oneside]{amsart}   	
\usepackage{amsmath,amsthm,amscd,amsfonts, amssymb, mathrsfs}
\usepackage{geometry}                		
\usepackage{graphicx}				
\usepackage{amssymb}
\newcommand{\sect}[1]{\section{#1}\setcounter{equation}{0}}

\font\mbn=msbm10 scaled \magstep1
\font\mbs=msbm7 scaled \magstep1
\font\mbss=msbm5 scaled \magstep1
\newfam\mbff
\textfont\mbff=\mbn
\scriptfont\mbff=\mbs
\scriptscriptfont\mbff=\mbss   

\newcommand{\Di}      {\mathbb{D}}

\newcommand{\N}       { \mathbb{N}}
\newcommand{\Z}        {\mathbb{Z}  }  
\newcommand\Co           {{\mathbb C}}

\newtheorem{Th}{Theorem}[section]
\newtheorem{Lm}[Th]{Lemma}

\newtheorem{Prop}[Th]{Proposition}
\newtheorem{R}[Th]{Remark}

\newtheorem*{Th A}{Theorem A}
\newtheorem*{Th B}{Theorem B}
\begin{document}

\title[On Stable Rank of $H^\infty$ on Coverings of Finite Bordered Riemann Surfaces]{On Stable Rank of ${\mathbf H^\infty}$ on Coverings of Finite Bordered Riemann Surfaces}
\author{Alexander Brudnyi}
\address{Department of Mathematics and Statistics\newline
\hspace*{1em} University of Calgary\newline
\hspace*{1em} Calgary, Alberta, Canada\newline
\hspace*{1em} T2N 1N4}
\email{abrudnyi@ucalgary.ca}

\keywords{Bass stable rank, $B$-ring, bounded holomorphic function, unbranched covering, elementary matrix}
\subjclass[2010]{Primary 30H50. Secondary 46J10.}

\thanks{Research is supported in part by NSERC}

\begin{abstract}
We prove that the Bass stable rank of the algebra of bounded holomorphic functions on an unbranched covering of a finite bordered Riemann surface is equal to one.
\end{abstract}

\date{}

\maketitle

\sect{Formulation of Main Results}
Let $S'$ be a (not necessarily connected) unbranched covering of a finite bordered Riemann surface $S$. In this paper we continue the study initiated in \cite{Br2} of the algebra $H^\infty(S')$ of bounded holomorphic functions on $S'$. (We write $H^\infty:=H^\infty(\Di)$, where $\Di\subset\Co$ is the open unit disk.) It was shown in our previous work that algebras $H^\infty(S')$ and $H^\infty$ share many common properties (e.g., they are Hermite, their maximal ideal spaces are two-dimensional with vanishing second \v{C}ech cohomology groups, etc., see \cite{Br2}--\cite{Br4} for the corresponding results). The purpose of this paper is to prove that these algebras have also the same Bass stable rank.
The latter notion is defined as follows.

Let $A$ be an associative ring with unit. For a natural number $n$ let
$U_n(A)$ denote the set of {\em unimodular} elements of $A^n$, i.e.,
\[
U_n(A)=\left\{(a_1,\dots, a_n)\in A^n\, :\,  Aa_1+\cdots +Aa_n=A\right\}.
\]
An element $(a_1,\dots, a_n)\in U_n(A)$ is called {\em reducible} if there exist $c_1,\dots, c_{n-1}\in A$ such that
$
(a_1+c_1 a_n,\dots, a_{n-1}+c_{n-1}a_n)\in U_{n-1}(A).
$
The {\em stable rank} ${\rm sr}(A)$ is the least $n$ such that every element of $U_{n+1}(A)$ is reducible. 
The concept  of the stable rank introduced by Bass \cite{B} plays an important role in some stabilization problems of algebraic $K$-theory.  
Following Vaserstein \cite{V2} we call a ring of stable rank $1$ a $B$-ring. (We refer to this paper for some examples and properties of $B$-rings.) 

In  \cite{T} Treil proved the following result.
\begin{Th A}
Let $f,g\in H^\infty$, $\|f\|_{H^\infty}\le 1$, $\|g\|_{H^\infty}\le 1$ and
\begin{equation}\label{treil1}
\inf_{z\in\Di}(|f(z)|+|g(z)|)=:\delta>0.
\end{equation}
Then there exists a function $G\in H^\infty$ such that the function $\Phi=f+gG$ is invertible in $H^\infty$, and moreover $\|G\|_{H^\infty}\le C$, $\|\Phi^{-1}\|_{H^\infty}\le C$, where the constant $C$ depends only on $\delta$.
\end{Th A}
\noindent (Here and below for a normed space $B$ its norm is denoted by $\|\cdot\|_B$.)

By the Carleson corona theorem condition \eqref{treil1} is satisfied if and only if $(f,g)\in U_2(H^\infty)$. Hence, Treil's theorem implies that $H^\infty$ is a $B$-ring. 

Theorem A was used by Tolokonnikov \cite{To} to prove that algebras $H^\infty(U)$ are $B$-rings for finitely connected domains and some Behrens domains $U$. Until now no other classes of Riemann surfaces $U$ for which $H^\infty(U)$ are $B$-rings were known.  In the present paper, we prove the following extension of Theorem A.
\begin{Th}\label{te1}
Let $S'$ be an unbranched covering of a finite bordered Riemann surface $S$. Let $f,g\in H^\infty(S')$, $\|f\|_{H^\infty(S')}\le 1$, $\|g\|_{H^\infty(S')}\le 1$ and
\begin{equation}\label{treil2}
\inf_{z\in S'}(|f(z)|+|g(z)|)=:\delta>0.
\end{equation}
Then there exists a function $G\in H^\infty(S')$ such that the function $\Phi=f+gG$ is invertible in $H^\infty(S')$, and moreover $\max\bigl\{\|G\|_{H^\infty(S')},\|\Phi^{-1}\|_{H^\infty(S')}\bigr\}\le C$, where the constant $C$ depends only on $\delta$ and $S$.
\end{Th}
By the corona theorem for $H^\infty(S')$ (see \cite[Cor.\,1.6]{Br2})
condition \eqref{treil2} is satisfied if and only if $(f,g)\in U_2(H^\infty(S'))$. Hence, Theorem \ref{te1} implies 
\begin{Th}\label{cor1}
$H^\infty(S')$ is a $B$-ring.
\end{Th}
\begin{R}
{\rm It is known that every $B$-ring is {\em Hermite} (see, e.g., \cite[Thm.\,2.7]{V2}),
 i.e., any finitely generated stably free right module over the ring is free (equivalently, any rectangular left-invertible matrix over the ring can be extended to an invertible matrix). 
Let $J\subset H^\infty(S')$ be a closed ideal and $H_J^\infty:=\{c+f\, :\, c\in\Co,\ f\in J\}$ be the unital closed subalgebra generated by $J$. Then Corollary \ref{cor1} implies that $H_J^\infty$ is a $B$-ring (see, e.g., \cite[Thm.\,4]{V1}); hence, it is Hermite. 
This gives a generalization of \cite[Thm.\,1.1]{Br3} proved by a different method.
}
\end{R}
Let $M_n(H^\infty(S'))$ be the algebra of $n\times n$ matrices with entries in $H^\infty(S')$ regarded as the subspace of bounded linear operators on $(H^\infty(S'))^n$ equipped with the operator norm.
We use Theorem \ref{te1} to describe the structure of the group $SL_n(H^\infty(S'))\subset M_n(H^\infty(S'))$ of matrices with determinant $1$. 

Recall that a matrix in  $SL_n(H^\infty(S'))$ is {\em elementary} if it differs from the identity matrix by at most one non-diagonal entry.
\begin{Th}\label{te2}
Every matrix in $SL_n(H^\infty(S'))$ of norm $\le M$
 is a product of at most $(n -1)(\frac{3n}{2} + 1)$ elementary
matrices whose norms are bounded from above by a constant depending only on $M$, $n$ and $S$.
\end{Th}
This result is new even for matrices with entries in $H^\infty$.

The proof of Theorem \ref{te1} is based on Theorem A and some results of the author presented in \cite{Br4} and \cite{Br5} along with some topological results.
In the next section we collect some results required for the proof of Theorem \ref{te1}.
The proof  is given in Section~4.

\sect{Auxiliary Results}
\noindent {\bf 2.1.} Let $\mathfrak M(A)$ denote the maximal ideal space of a commutative complex unital Banach algebra $A$, i.e., the set of nonzero homomorphisms $A \rightarrow\Co$ equipped with the {\em Gelfand topology}. In this part we present some facts about the maximal ideal space $\mathfrak M(H^\infty(S'))$, where $r:S'\rightarrow S$ is a (not necessarily connected but second-countable) unbranched covering of a bordered Riemann $S$, see \cite[Sect.\,2]{Br3}, \cite[Sect.\,4]{Br4} for details. \smallskip

Recall that $H^\infty(S')$ separates points of $S'$ and the map $\iota: S'\rightarrow \mathfrak M(H^\infty(S'))$ sending $x\in S'$ to the evaluation functional $\delta_x\in (H^\infty(S'))^*$ at $x$
 embeds $S'$ into $\mathfrak M(H^\infty(S'))$ as an open dense subset (-- the corona theorem for $H^\infty(S')$). \smallskip

The covering $r: S'\rightarrow S$  can be viewed as a fiber bundle over $S$ with a discrete (at most countable) fiber $F$. Let $E(S,\beta F)$ be the space obtained from $S'$ by taking the Stone-\v{C}ech compactifications of fibres under $r$.  It is a normal Hausdorff space  and $r$ extends to a continuous map $r_E: E(S,\beta F)\rightarrow S$ such that $\bigl(E(S,\beta F),S,r_E,\beta F\bigr)$ is a fibre bundle on $S$ with fibre $\beta F$ and $S'$ is an open dense subbundle
of $E(S,\beta F)$. Each  $f\in H^\infty(S')$ admits an extension $\hat f\in C(E(S,\beta F))$ and the algebra formed by such extensions separates points of $E(S,\beta F)$. Thus $\iota$ extends to a continuous injection $\hat{\iota}: E(S,\beta F)\rightarrow \mathfrak M(H^\infty(S'))$, $(\hat{\iota}(\xi))(f):=\hat f(\xi)$.\smallskip

In what follows, we identify $E(S,\beta F)$ with its image under $\hat\iota$. Also, for $K\subset S$ we set $K':=r^{-1}(K)$, $K_E:=r_E^{-1}(K)$ and for a subset $U$ of a topological space we denote
by $\mathring U$, $\bar U$ and $\partial U$ its interior, closure and boundary.\smallskip

It is well known that $S$ can be regarded as a domain in a compact Riemann surface $R$ such that
 $R\setminus \bar S$ is the finite disjoint union of open disks with analytic boundaries.
 Let  $A(S)\subset H^\infty(S)$ be the subalgebra of functions continuous up to the boundary.
 We denote by $\hat r:\mathfrak M(H^\infty(S'))\rightarrow\bar S$  the continuous surjective map induced by the transpose of the homomorphism
$A(S)\rightarrow H^\infty(S')$, $f\mapsto f\circ r$. Then  $E(S,\beta F)$ coincides
with the open set $\hat r^{-1}(S)$ and  $\hat r|_{E(S,\beta F)}=r_E$.
\smallskip

Let $U\subset R$ be open such that $V:=U\cap \bar S\neq\emptyset$.
Then $\hat r^{-1}(V)$ is an open subset of  $\mathfrak M(H^\infty(S'))$ and 
due to the corona theorem $\mathring{V}':=r^{-1}(\mathring V)$, $\mathring V:=U\cap S$, is an open dense subset  of $\hat r^{-1}(V)$.
\begin{Prop}\label{prop2.1}
Each $f\in H^\infty(\mathring{V}')$ admits an extension  $\hat f\in C(\hat r^{-1}(V))$.
\end{Prop}
\begin{proof}
We reduce the statement to some known results proved earlier by the author.\smallskip

 We have to extend $f$ continuously to each point $\xi\in \hat r^{-1}(V)$.
The set $\hat r^{-1}(V)$ is the disjoint union of the open set $\mathring V_E=\hat r^{-1}(\mathring V)$ and the set $\hat r^{-1}(V\cap\partial S)$. So we consider two cases.\smallskip

\noindent (1)\quad $\xi\in \hat r^{-1}(\mathring V)$. \smallskip

 Let $O\subset\mathring{V}$ be an open simply connected neighbourhood of $\hat r(\xi)$. By the definition of the bundle $E(S,\beta F)$, the set  $O_E= r_E^{-1}(O)$ is homeomorphic to 
$O\times\beta F$ and this homeomorphism maps $O'=r^{-1}(O)$ biholomorphically onto $O\times F$. Then Lemma 3.1 of \cite{Br1} implies that $f|_{O'}\in H^\infty(O')$ admits an extension $\hat f\in C(O_E)$ as required (because $O_E$ is an open neighbourhood of $\xi$).  \smallskip

\noindent (2)\quad $\xi\in\hat r^{-1}(V\cap\partial S)$.\smallskip

Let $\hat r(\xi)$ belong to a connected component $\gamma$ of $\partial S$.
By the definition of $S$, there are a relatively open neighbourhood $A_\gamma\subset \bar S$ of   $\gamma$  and  a homeomorphic map $A_\gamma\rightarrow A:=\{z\in\Co\, :\, c<|z|\le 1\}$, $c>0$, which maps $\gamma$ onto the unit circle $\mathbb S\subset\Co$ and is holomorphic on $\mathring A_\gamma$.  Without loss of generality we identify $A_\gamma$ with $A$ and $\gamma$ with $\mathbb S$. Then since $V\cap A\ne\emptyset$, there is a relatively open subset $\Pi\subset V\cap A$ which is a rectangle in polar coordinates with one side of the boundary on $\mathbb S$ such that $\hat r(\xi)\in \Pi$. Repeating literally the arguments of the proof of \cite[Prop.\,4.2]{Br4}, we obtain that each function from $H^\infty(\mathring{\Pi}')$ admits a continuous extension to $\hat r^{-1}(\Pi)$. Since the latter is an open neighbourhood of $\xi$, this gives the required extension of $f$ to $\xi$. We leave the details to the readers.
\end{proof}
\begin{R}\label{rem2.2}
{\rm Since $\mathring{V}'$ is dense in $\hat r^{-1}(V)$, the above extension  preserves supremum norm. Then the transpose of the restriction homomorphism $H^\infty(S')\rightarrow H^\infty(\mathring{V}')$, $f\mapsto f|_{\mathring{V}'}$,  induces a continuous map $s_V: \mathfrak M(H^\infty(\mathring{V}'))\rightarrow \mathfrak M(H^\infty(S'))$ with image $\hat r ^{-1}(\bar V)$ one-to-one on $s_V^{-1}(\hat r ^{-1}(V))$.}
\end{R}
\noindent {\bf 2.2.} 
A compact subset $K\subset\mathfrak M(H^\infty(S'))$ is said to be {\em holomorphically convex} (with respect to the algebra $H^\infty(S')$) if for every $\xi\not\in K$ there is $f\in H^\infty(S')$ such that 
\[
\max_K |\hat f|< |\hat f(\xi)|;
\]
here $\hat{f}\in C(\mathfrak M(H^\infty(S'))$ is the Gelfand transform of $f$.

A holomorphically convex subset $Z\subset\mathfrak M(H^\infty(S'))$ is called
a {\em hull} if there is a proper ideal $I\subset  H^\infty(S')$ such that
\[
Z=\{\xi\in \mathfrak M(H^\infty(S')\, :\, \hat f(\xi)=0\quad \forall f\in I\}.
\]

The algebra $H^\infty(S')$ is a $B$-ring if and only if for every hull
$Z\subset \mathfrak M(H^\infty(S')$ the map $C(\mathfrak M(H^\infty(S')),\Co^*) \rightarrow C(Z, \Co^*)$, $\Co^*:=\Co\setminus\{0\}$, induced by restriction
to $Z$ is onto, see \cite{CS}.
\smallskip

In the next two lemmas, $S=\Di$ and $S'=S\times\N$ (- the countable disjoint union of open unit disks).

\begin{Lm}\label{lem2.3}
If $K\subset\mathfrak M(H^\infty(S'))$ is holomorphically convex, then for every $g\in C(K,\Co^*)$, there exists $\tilde g\in C(\mathfrak M(H^\infty(S')),\Co^*)$ such that
$\tilde g|_K=g$.
\end{Lm}
\begin{proof}
According to \cite[Lm.\,5.3]{Br5} the homomorphism of the \v{C}ech cohomology groups $H^1(\mathfrak M(H^\infty(S')),\Z) \rightarrow H^1(K, \Z)$
induced by the restriction map to $K$ is surjective. In turn, by the Arens-Royden theorem $H^1(K,\Z)$ and $H^1(\mathfrak M(H^\infty(S')), \Z)$ are connected components of topological groups $C(K,\Co^*)$ and $C(\mathfrak M(H^\infty(S')),\Co^*)$, respectively. Hence, for each $g\in C(K,\Co^*)$, there is $g_1\in C(\mathfrak M(H^\infty(S')),\Co^*)$ such that $g\cdot (g_1^{-1})|_K=e^h$ for some $h\in C(K)$. Let $\tilde h\in C(\mathfrak M(H^\infty(S'))$ be an extension of $h$ (existing by the Titze-Urysohn theorem). Then $\tilde g=g_1e^{\tilde h}$ is the required extension of $g$.
\end{proof}
\begin{Lm}\label{lem2.4}
Suppose $K\subset\mathfrak M(H^\infty(S'))$ is holomorphically convex and $Z\subset\mathfrak M(H^\infty(S'))$ is a hull. Then $K\cup Z\subset\mathfrak M(H^\infty(S'))$ is holomorphically convex.
\end{Lm}
\begin{proof}
Let $\xi\not\in K\cup Z$. By the hypothesis, there exist $f,g\in H^\infty(S')$ such that $\hat f(\xi)=\hat g(\xi)=1$ and
\[
\max_K |f|=:c<1,\qquad g|_Z=0.
\]
Let $M:=\max_K |g|$. We choose $n\in\N$ such that $c^nM<1$. Then for $h:=f^n g\in H^\infty(S')$ we have
\[
\max_K |\hat h|\le c^nM<1=|\hat h(\xi)|
\]
This shows that the set $K\cup Z$ is holomorphically convex.
\end{proof}

\noindent {\bf 2.3.} For the proof of Theorem \ref{te1} we require the following topological result.
\begin{Lm}\label{lem2.5}
There is a finite cover $(U_j)_{j=1}^m$ of $\bar{S}$ by compact subsets homeomorphic to $\bar{\Di}$ such that each $U_i$ is contained in an open simply connected set $V_i\subset R$ with simply connected intersection $V_i\cap S$, each $U_i$ intersects with at most  two other sets of the family and each non-void $U_i\cap U_j$ is homeomorphic to $I:=[0,1]$.
\end{Lm}
\begin{proof}
Since $\bar S$ is triangulable, we may regard it as a two dimensional polyhedral manifold. It follows from the Whitehead theorem \cite[Thm.\,(3.5)]{W} that there are a (finite) one-dimensional polyhedron  $L\subset\bar{S}$ with sets of edges $E_L$ and vertices $V_L$ and a piecewise linear strong deformation 
retraction $F: \bar S\times I\rightarrow S$ of $\bar S$ onto $L$ such that

(a) $F^{-1}(x,1)\subset\bar{S}$ is a connected polyhedron homeomorphic to a {\em star tree} with internal vertex $x$ of degree $2$ if either $x\in\mathring e$ for some  $e\in E_L$ or $x\in V_L$ is of degree $\le 2$ and of degree $>2$
if $x\in V_L$ is of degree $>2$, and this homeomorphism maps $F^{-1}(x,1)\cap\partial S$ onto the set of external points of the tree.

(b) If $e\in E_L$, then $F^{-1}(\mathring{e},1)\cap\partial S$ is the disjoint union of two sets homeomorphic to $I$.

Let $E_L:=\{e_1,\dots,e_m\}$.
We define
\[
U_i:=\overline{F^{-1}(\mathring{e_i},1)},\quad 1\le i\le m.
\]
Then every $U_i$ is a polyhedral submanifold of $\bar S$ homeomorphic to $\bar\Di$ with the boundary formed by some arcs in  $\partial S$ along with some subsets of $F^{-1}(v_{i_j},1)$, $j=1,2$, homeomorphic to $I$; here $v_{i_1},v_{i_2}\in V_L$ are endpoints of $e_i$. Clearly every non-void intersection $U_i\cap U_j\subset F^{-1}(e_i\cap e_j,1)$ is homeomorphic to $I$. 
Moreover, it is readily seen that each $U_i$ is contained in an open simply connected subset
$V_i\subset R$ with simply connected intersection $V_i\cap S$ because $\bar S$ is the strong deformation retract of some of its open neighbourhoods in $R$ (see, e.g., \cite[Thm.\,(3.3)]{W}).
\end{proof}

\sect{Proof of Theorem \ref{cor1} }
We retain notation of Lemma \ref{lem2.5}. We set
 \[
 \partial U_i^\circ=\overline{\partial U_i\setminus\partial S}\quad {\rm and}\quad
 W_i:=V_i\cap S,\quad 1\le i\le m.
 \]
Then $\partial U_i^\circ$  consists of two connected components homeomorphic to $I$ and $W_i$ is an open simply connected subset of $S$.  By the definition,
$\partial U_i^\circ\subset\overline{W}_i$.

Let  $A(W_i)\subset H^\infty(W_i)$ be the subalgebra of functions continuous up to the boundary.
 We denote by $\hat r_i:\mathfrak M(H^\infty(W_i'))\rightarrow\overline{W}_i$  the continuous surjective map induced by the transpose of the homomorphism
$A(W_i)\rightarrow H^\infty(W_i')$, $f\mapsto f\circ r_i$. 

Let $K$ be either $\partial U_i^\circ$ or its connected component. We set
\[
\widetilde K:=\hat r_i^{-1}(K). 
\]
\begin{Lm}\label{lem3.1}
The set $\widetilde K\subset\mathfrak M(H^\infty(W_i'))$ is holomorphically convex.
\end{Lm}
\begin{proof}
By our construction the open set $V_i\setminus K$ is connected.
By the Riemann mapping theorem there is a biholomorphic map $\psi_i$ of  $V_i$ onto $\Di$. Then $\Di\setminus\psi_i(K)$ is a connected open subset of 
$\Di$. This implies that the compact set $\psi_i(K)\subset\Co$ is polynomially convex. Hence, $K\Subset V_i$ is holomorphically convex with respect to the algebra $H^\infty(V_i)$ and so it is holomorphically convex in $\overline{W}_i$ with respect to the algebra $A(W_i)$. Since $\hat r_i$ is a surjection onto $\overline{W}_i$ and $\widetilde K\subset\mathfrak M(H^\infty(W_i'))$ is the preimage of $K$, it is  holomorphically convex.
\end{proof}

Due to  Remark \ref{rem2.2} the transpose of the restriction homomorphism $H^\infty(S')\rightarrow H^\infty(W_i')$  induces a continuous map $s_i: \mathfrak M(H^\infty(W_i'))\rightarrow \mathfrak M(H^\infty(S'))$ with image $\hat r ^{-1}(\overline{W}_i)$ one-to-one on $s_i^{-1}(\hat r^{-1}(V_i\cap \bar S))$.

Let $Z\subset\mathfrak M(H^\infty(S'))$ be a hull and $g\in C(Z,\Co^*)$. To prove the theorem we have to extend $g$ to a function $\tilde g\in C(\mathfrak M(H^\infty(S')),\Co^*)$, see Section~2.2 above.

Clearly $Z_i:=s_i^{-1}(Z)$ is a hull for the algebra $H^\infty(W_i')$ and $s_i^*g\in C(Z_i,\Co^*)$. Since $W_i':=r^{-1}(W_i)$ is biholomorphic to $\Di\times F$, the Treil theorem implies that there is $g_i\in C(\mathfrak M(H^\infty(W_i')),\Co^*)$ which extends $s_i^*g\, (:=g\circ s_i)$. Hence $\tilde g_i:=g_i\circ s_i^{-1}\in C(\hat r^{-1}(V_i\cap\bar S),\Co^*)$ extends 
$g|_{Z\cap \hat r^{-1}(U_i)}$ (because $U_i\subset V_i\cap\bar S$). If $Z\cap \hat r^{-1}(U_i)=\emptyset$, we define $\tilde g_i=1$.

Next, we order the sets of the cover $(U_i)_{i=1}^m$ as follows. Choose some $U_{i_1}\subset\{U_1,\dots, U_m\}$.
If $U_{i_p}$ is already chosen, we choose $U_{i_{p+1}}$ so that 
\[
U_{i_{p+1}}\cap (\cup_{j=1}^p U_{i_j})\ne\emptyset .
\]
This is always possible because $\bar S$ is a connected set. We extend $g$ by induction on the indices of the order. 

For $j=1$ we set $\tilde g=\tilde g_{i_1}$ on $\hat r^{-1}(U_{i_1})$. Suppose that
$\tilde g$ is already defined on $\cup_{j=1}^p \hat r^{-1}(U_{i_j})$. Let us define it on $\cup_{j=1}^{p+1} \hat r^{-1}(U_{i_j})$. To this end let 
\[
g_{p,p+1}:=\tilde g \tilde g_{p+1}^{-1}\quad {\rm on}\quad \hat r^{-1}(\cup_{j=1}^p U_{i_j})\cap\hat r^{-1}(U_{i_{p+1}}).
\]
By the definition, the above intersection, say $X$, is either the preimage  of $\partial U_{i_{p+1}}^o$ or  its connected component under $\hat r$. Due to Lemma \ref{lem3.1}, the set $s_{i_{p+1}}^{-1}(X)\cup Z_{i_{p+1}}$ is holomorphically convex with respect to $H^\infty(W_{i_{p+1}})$. Moreover, $s_{i_{p+1}}^*(g_{p,p+1})\in C(s_{i_{p+1}}^{-1}(X),\Co^*)$ and equals $1$ on $s_{i_{p+1}}^{-1}(X)\cap Z_{i_{p+1}}$. Hence, it can be extended to a function in $C(s_{i_{p+1}}^{-1}(X)\cup Z_{i_{p+1}},\Co^*)$ attaining value $1$ on $Z_{i_{p+1}}$.
Due to Lemma \ref{lem2.3} the extended function can further be extended to a function from 
$C(s_{i_{p+1}}^{-1}(\hat r^{-1}(U_{i_{p+1}})),\Co^*)$. Composing this extension with $s_{i_{p+1}}^{-1}$ we obtain an extension $\tilde g_{p,p+1}$ of $g_{p,p+1}$  equal to $1$ on $Z\cap \hat r^{-1}(U_{i_{p+1}})$. Let us define
\[
\tilde g|_{\hat r^{-1}(U_{i_{p+1}})}:=\tilde g_{p+1}\tilde g_{p,p+1}.
\]
Then $\tilde g|_{\hat r^{-1}(U_{i_{p+1}})}$ extends $g|_{Z\cap\hat r^{-1}(U_{i_{p+1}})}$ and 
\[
\tilde g|_{\hat r^{-1}(U_{i_{p+1}})}\cdot\tilde g^{-1}|_{\hat r^{-1}(\cup_{j=1}^p U_{i_j})}=\tilde g_{p+1}\tilde g_{p,p+1}\tilde g^{-1}=1
\quad {\rm on}\quad \hat r^{-1}(\cup_{j=1}^p U_{i_j})\cap\hat r^{-1}(U_{i_{p+1}}),
\]
i.e., $\tilde g|_{\hat r^{-1}(U_{i_{p+1}})}$ is the required extension of $\tilde g|_{\hat r^{-1}(\cup_{j=1}^p U_{i_j})}$ to $\cup_{j=1}^{p+1} U_{i_j}$. This completes the proof of the induction step and hence of the theorem.

\sect{Proofs of Theorems \ref{te1} and \ref{te2}}
\begin{proof}[Proof of Theorems \ref{te1}]
Without loss of generality we may assume that $S'$ is a connected unbranched covering of $S$.  Let $f,g\in H^\infty(S')$, $\|f\|_{H^\infty(S')}\le 1$, $\|g\|_{H^\infty(S')}\le 1$ and
\begin{equation}\label{eq4.1}
\inf_{z\in S'}(|f(z)|+|g(z)|)=:\delta>0.
\end{equation}
Due to Theorem \ref{cor1} there exists a function $G\in H^\infty(S')$ such that the function $f+gG$ is invertible in $H^\infty(S')$.
By $\mathcal G_{f,g,\delta,S'}$ we denote the class of such functions $G$.
We have to prove that
\begin{equation}\label{eq4.2}
C=C(\delta,S):=\sup_{f,g,S'}\inf_{G\in \mathcal G_{f,g,\delta,S'}} 
\max\bigl\{\|G\|_{H^\infty(S')},\|(f+gG)^{-1}\|_{H^\infty(S')}\bigr\}
\end{equation}
is finite. (Here the first supremum is taken over all functions $f,g$ satisfying the above hypotheses and all connected unbranched coverings $S'$ of $S$.)

Let $\{S_i'\}_{i\in\N}$ and $\{f_i\}_{i\in\N},\{g_i\}_{i\in\N}$, $f_i,g_i\in H^\infty(S_i')$, be sequences satisfying assumptions of the theorem such that
\begin{equation}\label{eq4.3}
C=\lim_{i\rightarrow\infty}\inf_{G\in \mathcal G_{f_i,g_i,\delta,S_i'}} 
\max\bigl\{\|G\|_{H^\infty(S_i')},\|(f_i+g_iG)^{-1}\|_{H^\infty(S_i')}\bigr\}.
\end{equation}
The disjoint union $S':=\sqcup_{i\in\N}\, S_i'$ is clearly an unbranched covering of $S$ and
functions $f,g\in H^\infty(S')$ defined by the formulas
\[
f|_{S_i'}:=f_i,\quad g|_{S_i'}:=g_i,\quad i\in\N,
\]
are of norms $\le 1$ and satisfy condition \eqref{eq4.1} on $S'$.
Then due to Theorem \ref{cor1} there exists a function $G\in H^\infty(S')$ such that the function $f+gG$ is invertible in $H^\infty(S')$.
We set
\[
G_i:=G|_{S_i'},\quad i\in\N.
\]
Then due to \eqref{eq4.3}
\[
C\le \sup_{i\in\N}\max\bigl\{\|G_i\|_{H^\infty(S_i')},\|(f_i+g_iG_i)^{-1}\|_{H^\infty(S_i')}\bigr\}=\max\bigl\{\|G\|_{H^\infty(S')},\|(f+gG)^{-1}\|_{H^\infty(S')}\bigr\}.
\]
This completes the proof of the theorem.
\end{proof}
\begin{proof}[Proof of Theorem \ref{te2}]
According to \cite[Thm.\,1.3\,(b)]{Br4} the covering dimension of the maximal ideal space $\mathfrak M(H^\infty(S'))$ is 2. In turn, due to the Browder theorem \cite[Thm.\,6.11]{Bro} the second homotopy group $\pi_2(SL_n(\Co))=0$ (here $SL_n(\Co)$ is the group of $n\times n$ complex matrices with determinant 1). These two facts and the Hu theorem \cite[(11.4)]{Hu} imply that the homotopy classes of the continuous mappings $f:\mathfrak M(H^\infty(S'))\rightarrow SL_n(\Co)$ are in a one-to-one correspondence with the elements of the first \v{C}ech cohomology group $H^1(\mathfrak M(H^\infty(S')),\pi_1(SL_n(\Co)))$. This group is trivial because the space $SL_n(\Co)$ is simply connected. Hence, each $f\in C(\mathfrak M(H^\infty(S')), SL_n(\Co))$ is homotopic to a constant map with value $1_n$ (the unit of $SL_n(\Co)$), i.e., the space  $C(\mathfrak M(H^\infty(S')), SL_n(\Co))$ is connected. Next, by the Arens theorem \cite{A} the Gelfand transform induces a bijection between the sets of connected components of the spaces $SL_n(H^\infty(S'))$ and $C(\mathfrak M(H^\infty(S')), SL_n(\Co))$. Therefore the group $SL_n(H^\infty(S'))$ is connected as well. In particular, each matrix in $SL_n(H^\infty(S'))$
can be presented as a finite product of elementary matrices. (This is a well-known fact; it can be deduced, e.g., from \cite[Thm.\,2.1]{BRS}.) Then since $H^\infty(S')$ is a $B$-ring (by Theorem \ref{cor1}),  Lemma 9 and Remark 10 of \cite{DV} imply that 
each matrix $F\in SL_n(H^\infty(S'))$ can be presented as a  product of at most $(n -1)(\frac{3n}{2} + 1)$ elementary matrices. Let us show that if 
\begin{equation}\label{eq4.4}
\|F\|_{M_n(H^\infty(S'))}\le M,
\end{equation}
these matrices would be chosen so that their norms were bounded from above by a constant depending only on $M$, $n$ and $S$.

As before we may assume that $S'$ is connected. Let $\mathcal F_{M,S',n}$ be the class of matrices $F\in SL_n(H^\infty(S'))$ satisfying \eqref{eq4.4}.  For every $F\in\mathcal F_{M, S',n}$ by $\Pi_{F,M,S',n}$ we denote the set of all possible products of $F$
by at most $(n -1)(\frac{3n}{2} + 1)$ elementary matrices. By the above arguments the set $\Pi_{F,M,S',n}$ is non-void. For each $\pi\in\Pi_{F,M,S',n}$ by 
$\|\pi\|$ we denote maximum of norms of elementary matrices in $\pi$.
We  have to prove that
\begin{equation}\label{eq4.5}
C=C(S,M,n):=\sup_{S',F\in\mathcal F_{M,S',n
}}\inf_{\pi\in\Pi_{F,M,S',n}}\|\pi\|<\infty;
\end{equation}
here $S'$ runs over all connected unbranched coverings of $S$.

Let $S_i'$ and $F_i\in \mathcal F_{M,S_i',n}$, $i\in\N$, be  such that
\begin{equation}\label{eq4.6}
C=\lim_{i\rightarrow\infty}\inf_{\pi\in\Pi_{F_i,M,S_i',n}}\|\pi\|.
\end{equation}
It is clear that the disjoint union $S':=\sqcup_{i\in\N}\, S_i'$ is an unbranched covering of $S$ and
the matrix $F\in H^\infty(S')$ defined by the formula
\[
F|_{S_i'}:=F_i,\quad i\in\N,
\]
belongs to the class $\mathcal F_{M,S',n}$. Then there is $\pi\in\Pi_{F,M,S',n}$. By $\pi_i$ we denote the product obtained by restriction of elementary matrices in $\pi$ to $S_i'$. Then each $\pi_i\in \Pi_{F_i,M,S_i',n}$ and so
due to \eqref{eq4.6}
\[
C\le \sup_{i\in\N}\|\pi_i\|=\|\pi\|<\infty.
\]
This completes the proof of the theorem.
\end{proof}

\end{document}